\documentclass{amsart}

\usepackage{upgreek}
\usepackage{mathrsfs}
\usepackage{amsmath}
\usepackage{booktabs}
\usepackage{amsthm}
\usepackage{amssymb}
\usepackage{amsfonts}
\usepackage{xcolor}
\usepackage{enumerate}
\usepackage{CJK}
\usepackage{bbm}
\usepackage{tikz}
\usepackage{tikz-cd}
\usetikzlibrary{matrix,arrows,decorations.pathmorphing}

\begin{CJK}{UTF8}{}
\gdef\yama{\mbox{\textbf{山}}}

\gdef\ten{\mbox{\textbf{天}}}
\end{CJK}

\newcommand{\A}{\mathscr{A}}
\renewcommand{\P}{\mathbb{P}}

\newcommand{\F}{\mathbb{F}}

\newcommand{\Q}{\mathbb{Q}}

\newcommand{\bmu}{\mathbf{\upmu}}
\newcommand{\Z}{\mathbb{Z}}

\DeclareMathOperator{\Aut}{Aut}
\DeclareMathOperator{\Out}{Out}
\DeclareMathOperator{\Inn}{Inn}

\DeclareMathOperator{\ord}{ord}
\DeclareMathOperator{\Gal}{Gal}
\DeclareMathOperator{\Ind}{Ind}

\DeclareMathOperator{\Res}{Res}

\newtheorem{theorem}{Theorem}
\newtheorem*{main-thm}{Main Theorem}
\newtheorem{lemma}[theorem]{Lemma}
\newtheorem{proposition}[theorem]{Proposition}
\newtheorem{corollary}[theorem]{Corollary}

\theoremstyle{definition}

\numberwithin{equation}{section}
\numberwithin{theorem}{section}

\title[Abelian Surfaces good away from 2]{Abelian Surfaces good away from 2}

\author[C.~Rasmussen and A.~Tamagawa]{Christopher Rasmussen and Akio
  Tamagawa}

\begin{document}
\begin{CJK}{UTF8}{min}

\maketitle

\begin{abstract}
Fix a number field $k$ and a rational prime $\ell$. We consider abelian varieties whose $\ell$-power torsion generates a pro-$\ell$ extension of $k(\bmu_{\ell^\infty})$ which is unramified away from $\ell$. It is a necessary, but not generally sufficient, condition that such varieties have good reduction away from $\ell$. In the special case of $\ell = 2$, we demonstrate that for abelian surfaces $A/\Q$, good reduction away from $\ell$ does suffice. The result is extended to elliptic curves and abelian surfaces over certain number fields unramified away from $\{2, \infty \}$. An explicit example is constructed to demonstrate that good reduction is not sufficient, even at $\ell = 2$, for abelian varieties of sufficiently high dimension.
\end{abstract}

\section{Introduction}

Let $k$ be a number field, $\bar{k}$ a fixed algebraic closure, and $\ell$ a rational prime number. We let $G_k$ denote the absolute Galois group $\Gal(\bar{k}/k)$. We define two extensions of the function field $\bar{k}(t)$: $\tilde{M}$ is the maximal extension of $\bar{k}(t)$ which is unramified outside the places $t = 0, 1, \infty$, and  $M$ is the maximal sub-extension of $\tilde{M}/\bar{k}(t)$ which is pro-$\ell$. Then the Galois group
\[ \pi := \Gal(M/\bar{k}(t)) \]
may be identified with the pro-$\ell$ algebraic fundamental group of $\P^1_{01\infty} \times_k \bar{k}$. From the tower of function fields $k(t) \subset \bar{k}(t) \subset M$, we have the induced short exact sequence
\[
\begin{tikzcd}
1 \ar{r} & \pi \ar{r} & \Gal\bigl( M/k(t) \bigr) \ar{r} & G_k \ar{r} & 1,
\end{tikzcd}
\]
where $\Gal(\bar{k}(t)/k(t)) \cong G_k$. For any element $\sigma \in G_k$, we may lift to an element in $\Gal(M/k(t))$, and use this lift to conjugate the normal subgroup $\pi$. We may associate to each $\sigma$ an automorphism of $\pi$; since the lift of $\sigma$ is only well-defined up to an element of $\pi$, the resulting automorphism is only defined up to an inner automorphism of $\pi$. This gives the \emph{outer canonical pro-$\ell$ Galois representation}:
\[ \Phi_\ell \colon G_k \longrightarrow \Out(\pi) := \Aut(\pi)/\Inn(\pi). \]
We let $\yama := \yama(k, \ell)$ denote the field fixed by the kernel of $\Phi_\ell$, and let $\ten := \ten(k, \ell)$ denote the maximal pro-$\ell$ extension of $k(\bmu_{\ell^\infty})$ which is unramified away from $\ell$.

Let us explain the motivation behind this notation: The kanji character $\yama$ (pronounced \emph{yama}) means ``mountain,'' and the character $\ten$ (pronounced \emph{ten}) means ``heaven.'' It is known by work of Anderson and Ihara \cite{Anderson-Ihara:1988} that $\yama \subseteq \ten$. It is natural to ask in which cases $\yama$ and $\ten$ coincide -- Ihara first posed this question for $k = \Q$ in 1986, which may be loosely phrased as: \emph{When does the mountain ($\yama$) reach the heavens ($\ten$)}?

The recent result of Francis Brown \cite{Brown:2012}, which demonstrates the Deligne-Ihara conjecture for any prime $\ell$, implies that $\yama(\Q, \ell) = \ten(\Q, \ell)$ for any odd regular prime $\ell$. This implication was demonstrated in an earlier work by Sharifi \cite{Sharifi:2002}; the interested reader may also find further details in the article \cite{Ihara:2002}.

Let $A/k$ be an abelian variety of dimension $g > 0$. We say $A/k$ is \emph{heavenly at $\ell$} if $k(A[\ell^\infty]) \subseteq \ten(k,\ell)$. By \cite[\S1, Theorem 1]{Serre-Tate:1968}, if $A/k$ is heavenly at $\ell$, then it necessarily has good reduction away from $\ell$. When a curve $C/k$ has a Jacobian variety $J/k$ which is heavenly at $\ell$, it is sometimes possible to demonstrate $k(J[\ell^\infty]) \subseteq \yama$, via a combination of geometric (\cite{Rasmussen:2004}, \cite{Papanikolas-Rasmussen:2005}) or arithmetic (\cite{Rasmussen-Tamagawa:2008}) arguments. This connection is described in greater detail in \cite{Rasmussen-Tamagawa:2008, Rasmussen-Tamagawa:2012}. However, for fixed $k$, $g$ and $\ell$, there are only finitely many $k$-isomorphism classes of such abelian varieties.

For fixed $k$, $g$, and $\ell$, we let $\mathscr{G}(k,g,\ell)$ denote the set of $k$-isomorphism classes of abelian varieties $A/k$ of dimension $g$ which have good reduction away from $\ell$. This set is finite by the Shafarevich Conjecture (more precisely, by Zarhin's extension of Faltings' proof of the Shafarevich Conjecture to unpolarized abelian varieties). We denote by $\mathscr{A}(k,g,\ell)$ the subset of $\mathscr{G}(k,g,\ell)$ of isomorphism classes represented by abelian varieties which are heavenly at $\ell$. Under assumption of GRH, the authors have shown in \cite{Rasmussen-Tamagawa:2012} that for any choice of $k$ and $g$, the set $\mathscr{A}(k,g,\ell)$ is empty for sufficiently large $\ell$.

Curiously, in the case $k = \Q$, $g = 1$, we have the following observations:
\[ \begin{split}
\mathscr{A}(\Q,1,2) & = \mathscr{G}(\Q,1,2), \\
\mathscr{A}(\Q,1,3) & = \mathscr{G}(\Q,1,3), \\
\mathscr{A}(\Q,1,5) & = \mathscr{G}(\Q,1,5) = \varnothing, \\
\mathscr{A}(\Q,1,7) & = \mathscr{G}(\Q,1,7).
 \end{split} \]
(The coincidence is lost already at $\ell = 11$, where $\mathscr{A}(\Q,1,11)$ is a proper and nonempty subset of $\mathscr{G}(\Q,1,11)$.) In the present note, we show a similar phenomenon holds in dimension $2$ when $\ell = 2$, at least when we restrict attention to principally polarized abelian varieties. To be more precise, we introduce two additional sets of isomorphism classes, as follows:
\[
\begin{matrix}
\mathscr{G}^\mathrm{pp}(k,g,\ell) & \subseteq & \mathscr{G}(k,g,\ell) \\
\rotatebox{90}{$\subseteq$} & & \rotatebox{90}{$\subseteq$} \\
\mathscr{A}^\mathrm{pp}(k,g,\ell) & \subseteq & \mathscr{A}(k,g,\ell) \\
\end{matrix}
 \]
The elements of $\mathscr{G}^\mathrm{pp}(k,g,\ell)$ are classes represented by abelian varieties which are  principally polarized (over $k$). The set $\mathscr{A}^\mathrm{pp}(k,g,\ell)$ denotes the subset of classes represented by heavenly principally polarized abelian varieties.
\begin{main-thm}
Let $K_0/\Q$ be an extension unramified away from $\{2, \infty \}$ with $[K_0:\Q] \leq 2$.
Every principally polarized abelian surface $A/K_0$ with good reduction away from $2$ is heavenly at $2$. In other words,
\[ \mathscr{A}^\mathrm{pp}(K_0,2,2) = \mathscr{G}^\mathrm{pp}(K_0,2,2). \]
\end{main-thm}
Explicitly, $K_0$ is one of the fields $\Q(\sqrt{d})$, $d \in \{\pm 1, \pm 2 \}$.

The paper proceeds as follows. In \S2, we break the proof into three cases, determined by the structure of $A/K_0$, an abelian variety representing a class in $\mathscr{G}^\mathrm{pp}(K_0, 2, 2)$. In \S3, we collect some facts about extensions of $\Q$ unramified away from $2$, and use these observations to extend a previous result \cite{Rasmussen:2004}. The three cases of the main theorem are handled in detail in \S4; the question in higher dimensions is briefly explored in \S5.

\section{Outline of Proof}

Let $\mathcal{A}_2$ denote the moduli space of principally polarized abelian surfaces, and $\mathcal{M}_2$ the moduli space of curves of genus $2$. The Deligne-Mumford compactification of $\mathcal{M}_2$ contains an intermediate space, $\mathcal{M}_2^*$, corresponding to the locus of compact type. It is known that the locus of compact type surjects onto $\mathcal{A}_2$. Moreover, the extremal locus, $\mathcal{M}_2^* - \mathcal{M}_2$, corresponds exactly to those points in $\mathcal{A}_2$ represented by products of elliptic curves. Consequently, over an algebraically closed field $k$, an abelian surface $A/k$ must be isomorphic to either the Jacobian of a (smooth) genus $2$ curve, or a product of elliptic curves. When $k$ is not algebraically closed, we have the following theorem of Gonz\'alez-Gu\`ardia-Rotger \cite[Theorem 3.1]{Gonzalez-Guardia-Rotger:2005}:
\begin{theorem}
Let $k$ be a number field, and let $A/k$ be a principally polarized abelian surface. Then as a polarized abelian variety, $A$ is isomorphic over $k$ to one of the following:
\begin{enumerate}[(1)]
\item $J$, the Jacobian variety of $C/k$, a smooth curve of genus $2$,
\item $E_1 \times E_2$, where $E_i/k$ are elliptic curves,
\item $W := \mathrm{Res}_{k'/k} E$, where $W$ is the Weil restriction of the elliptic curve $E/k'$, and $k'/k$ is a quadratic extension.
\end{enumerate}
\end{theorem}
Gonz\'alez, Gu\`ardia and Rotger give the explicit polarization on $J$, $E_1 \times E_2$, or $W$; it is always the `natural choice.' In order to prove the main theorem, we consider each of the cases (1), (2), (3). In each case, we appeal to some previously established results on the nature of Galois extensions of $\Q$ unramified away from $\{2, \infty\}$ of small degree. These are mainly due to Harbater, Jones, and Jones-Roberts.

\section{Extensions unramified away from $2$}

The database of number fields of Jones and Roberts \cite{Jones-Roberts:DB}, together with a result of Jones \cite{Jones:2010}, describe all small degree extensions of $\Q$ which are unramified away from $\{2, \infty \}$:
\begin{theorem}\label{theorem:Jones}
Suppose $K/\Q$ is a finite extension unramified away from $\{2, \infty \}$, and $[K:\Q] < 16$. Let $L/\Q$ denote the Galois closure of $K$, and let $G = \Gal(L/\Q)$. Then:
\begin{enumerate}[(a)]
\item $[K:\Q] \in \{1, 2, 4, 8 \}$,
\item if $[K:\Q] = 4$, then $G$ is isomorphic to $V$, $\Z/4\Z$, or $D_4$,
\item if $[K:\Q] = 8$, then $G$ is a $2$-group of order at most $128$,
\item there exist $\sigma, \tau \in G$ such that $G = \langle \sigma, \tau \rangle$ and $\tau^2 = 1$.
\end{enumerate}
\end{theorem}
Jones's result incorporates several previous results by Harbater, Mark\v{s}a\u{\i}tis, Brueggeman, and Lesseni. We remark more concretely on (c). Several octic extensions of $\Q$ unramified away from $\{2, \infty \}$ have Galois closures of degree $64$ \cite{Jones-Roberts:DB}; the bound $2^7 = 128$ follows simply from the observation that $\ord_2 |S_8| = 7$. The following lemma improves (c) slightly.
\begin{lemma}\label{lemma:octic}
Suppose $K/\Q$ is an octic extension unramified away from $\{2, \infty \}$. Let $L$ be the Galois closure of $K$ in $\bar{\Q}$, a fixed algebraic closure of $\Q$. Then $[L:\Q] =2^\nu$, with $\nu \leq 6$.
\end{lemma}
\begin{proof}
Let $G = \Gal(L/\Q)$. By Theorem \ref{theorem:Jones} (c), we need only eliminate the possibility that $|G| = 128$, which we now suppose for the sake of contradiction. We may identify $G$ with some subgroup of $S_8$; note that $G$ must be a Sylow-$2$ subgroup of $S_8$. Consider the following subgroup $H \leq S_8$:
\[ H := \bigl\langle (1234), (13), (5678), (57), \tau \bigr\rangle, \]
where $\tau = (15)(26)(37)(48)$. (The idea is to `mix' two distinct copies of $D_4$ with $\tau$ to obtain $H$.) The subgroup $H$ is also a Sylow-$2$ subgroup of $S_8$; hence, $G \cong H$. However, a routine calculation verifies that $H$ cannot be generated by $2$ elements, let alone satisfy the condition (d) of Theorem \ref{theorem:Jones}. Thus, $G$ cannot either, and this gives the desired contradiction.
\end{proof}
Finally, we will need one more result on such extensions, in this case due to Harbater \cite[Theorem 2.25]{Harbater:1994}.
\begin{proposition}\label{prop:harbater}
Suppose $L/\Q$ is a Galois extension which is unramified away from $\{2, \infty \}$. If $[L:\Q] < 272$, then $[L:\Q]$ is a power of $2$.
\end{proposition}
As Harbater observes, this is the best possible bound, as there exists a subfield of $\Q(\bmu_{64})$ whose Hilbert class field gives a Galois extension $L/\Q$ unramified away from $\{2, \infty \}$ of exact degree $272$. (This field is discussed in more detail below.)

In \cite{Rasmussen:2004}, the first author proved $\mathscr{A}^\mathrm{pp}(\Q,1,2) = \mathscr{G}^\mathrm{pp}(\Q,1,2)$ through a geometric argument, by demonstrating a certain criterion of Anderson and Ihara holds for a representative of each class in $\mathscr{G}^\mathrm{pp}(\Q,1,2)$. Before turning to the main theorem, we give a more direct proof of this result, which also holds for some number fields other than $\Q$. We will also use this result in the next section.
\begin{proposition}\label{prop:ell_curve}
Suppose $K_0/\Q$ is an extension unramified away from $\{2, \infty \}$ and $[K_0:\Q] \leq 4$. Then  $\mathscr{A}^\mathrm{pp}(K_0,1,2) = \mathscr{G}^\mathrm{pp}(K_0,1,2)$.
\end{proposition}
\begin{proof}
Suppose $E/K_0$ is an elliptic curve with good reduction away from $2$. Let $L$ denote the Galois closure of $K_0(E[2])/\Q$. Since the tower
\[ \Q \subseteq K_0 \subseteq K_0(E[2]) \subset K_0(E[2^\infty]) \]
is unramified away from $\{2, \infty \}$ and the top extension is pro-$2$, it suffices to show that $[L : K_0]$ is a power of $2$. Let $d = [K_0(E[2]):K_0]$. Since the Galois group of this extension is isomorphic to a subgroup of $GL_2(\F_2)$, $d \in \{1, 2, 3, 6 \}$. If $d = 3$ or $d = 6$, then there exists a field $F$ with $K_0 \subset F \subseteq K_0(E[2])$ and $[F : K_0] = 3$. But now $[F:\Q] \leq 12$, and so by Theorem \ref{theorem:Jones}, $[F:\Q]$ is a power of $2$. Since it is also divisible by $3$, this is a contradiction. Thus, $d \leq 2$, and so $[K_0(E[2]):\Q] \leq 8$. By Theorem \ref{theorem:Jones}, we see that $[L:\Q]$ is a power of $2$. Consequently, $E$ is heavenly at $2$ and the result holds.
\end{proof}

\section{Abelian Surfaces}

We now turn to the proof of the Main Theorem. Before considering the three cases outlined in \S2, we treat specifically the case $K_0 = \Q$, where a stronger result may be obtained, and with less effort. Let $M$ be the maximal extension of $\Q$ which is unramified away from $\{2, \infty \}$, and set $\Delta := \Gal(M/\Q) = \pi_1(\Z[\tfrac{1}{2}])$.
\begin{proposition}
Let $\rho \colon \Delta \to GL_4(\F_2)$ be a Galois representation. Then the image of $\rho$ is a $2$-group.
\end{proposition}
\begin{proof}
The representation $\rho$ induces an action of $\Delta$ on $V := \F_2^{\oplus 4}$, and necessarily factors through $G := \Gal(M_0/\Q)$ for some finite Galois extension $M_0/\Q$. Selecting $M_0$ minimal, we have $G \cong \rho(\Delta)$. Further, there is an induced faithful action of $G$ on $V^\circ := V - \{0 \}$. For any $v \in V^\circ$, let $G_v$ denote the stabilizer of $v$, $M_v$ the subfield of $M_0$ fixed by $G_v$, and $L_v$ the Galois closure of $M_v/\Q$. As $M_0/\Q$ is Galois, we have $L_v \subseteq M_0$ for all $v$. For any $v \in V^\circ$, an application of the orbit-stabilizer theorem gives
\[ [M_v : \Q] = [G : G_v] = \#\mathrm{Orb}(v) \leq |V^\circ| < 16. \]
By Theorem \ref{theorem:Jones}, $[M_v : \Q] \mid 8$ and $L_v/\Q$ is a $2$-extension. The faithfulness of the action guarantees $\cap_v G_v = \{1 \}$, or in other words, that $M_0$ is the compositum of the $M_v$. Equivalently, $M_0$ is the compositum of the $L_v$, and so $M_0/\Q$ is a $2$-extension and $\rho(\Delta)$ is a $2$-group.
\end{proof}
\begin{corollary}
$\mathscr{G}(\Q,2,2) = \mathscr{A}(\Q,2,2)$.
\end{corollary}
\begin{proof}
Suppose $[A] \in \mathscr{G}(\Q,2,2)$, and let $\rho$ be the induced Galois representation on the $2$-torsion of $A$. Then in the context of the previous proposition, $M_0 = \Q(A[2])$ is a $2$-extension of $\Q$, and hence $[A] \in \mathscr{A}(\Q,2,2)$.
\end{proof}
\subsection{Case 1: Jacobians}
In this section, we prove the following proposition:
\begin{proposition}\label{prop:Jacobian}
Let $K_0/\Q$ be an extension unramified away from $\{2, \infty \}$, with $[K_0:\Q] \leq 2$. Let $C/K_0$ be a smooth projective curve of genus $2$, and let $J$ denote the Jacobian variety of $C$. Suppose $[J] \in \mathscr{G}^\mathrm{pp}(K_0,2,2)$. Then $[J] \in \mathscr{A}^\mathrm{pp}(K_0,2,2)$.
\end{proposition}
Let $C/K_0$ be a smooth projective curve of genus $2$. Then $C$ is hyperelliptic, and admits a degree $2$ morphism $C \to \P := \P^1_{K_0}$, which is unique up to coordinate change of $\P^1_{K_0}$. The ramification locus $S \subset C$ for $C \to \P$ (that is, the support of the sheaf $\Omega_{C/\P}$) has degree $6$. Namely, $\bar{S} := S \times_{K_0} \bar{K}_0$ consists of $6$ points. Further, $\bar{S}$ admits a natural action of $G_{K_0} := \Gal(\bar{K}_0/K_0)$, and the set of $G_{K_0}$-orbits of $\bar{S}$ is identified with $S$.

For a finite set $T$ and an integer $n \geq 0$, let $T_n$ denote the collection of subsets of $T$ of cardinality $n$; that is, $T_n := \{U \subseteq T : |U| = n \}$. The natural map $\Aut(T) \to \Aut(T_n)$ between symmetric groups is injective, if $0 < n < |T|$. With this notation established, it is a standard result that the map
\[ \bar{S}_2 \to J[2]-\{0\}, \qquad \{s,s'\} \mapsto \mathrm{cl}(s-s') \]
is well-defined, $G_{K_0}$-equivariant, and bijective (for further details, see \cite{Cassels-Flynn:1996}).
It follows that the compositum $K_0(S)$ of the residue fields $K_0(P)$ for $P\in S$ coincides with $K_0(J[2])$. Thus, under the assumption that $[J]\in\mathscr{G}^{\mathrm{pp}}(K_0,2,2)$, we have great control over the field $K_0(J[2])$ for certain fields $K_0$, by combining the observations of Jones and Harbater:
\begin{proposition}
Let $K_0/\Q$ be unramified away from $\{2, \infty \}$, with $[K_0:\Q] \leq 2$. Let $C/K_0$ be a smooth projective curve of genus $2$, and let $J$ denote the Jacobian of $C$. Let $K' = K_0(J[2])$, and let $L'$ be the Galois closure of $K'/\Q$. Suppose that $J/K_0$ has good reduction away from $2$. Then $L'/\Q$ is a $2$-extension satisfying $[L':\Q] \leq 256$.
\end{proposition}
\begin{proof}
We work in a fixed algebraic closure of $K_0$. Let $S = \{P_1,\dots,P_t\}$ be the ramification locus of the degree $2$ morphism $C \to \P$, as above. For $K_i := K_0(P_i)$, the extension $K_i/K_0$ is contained in $K_0(J[2])$, hence
necessarily unramified away from $\{2, \infty \}$. Let $d_i := [K_i : K_0]$. A priori, $d_i \leq 6$, and so $[K_i:\Q] \leq 2 d_i \leq 12$. By Theorem \ref{theorem:Jones}, we see $d_i \in \{1, 2, 4 \}$. Possibly after reindexing, we may assume $d_1 \geq \cdots \geq d_t$. The possible values of $\mathbf{d} = (d_1, \dots, d_t)$ are given in Table \ref{table:degrees}.
\begin{table}[t!]
\begin{tabular}{lrc}
\toprule
$t\;\;\;$ & $\mathbf{d} = (d_1, \dots, d_t)$ & Bound for $[\tilde{L}:\Q]$ \\
\midrule
$2$ & $(4, 2)$ & $2^8$ \\
$3$ & $(4, 1, 1)$ & $2^6$ \\
$3$ & $(2, 2, 2)$ & $2^7$ \\
$4$ & $(2, 2, 1, 1)$ & $2^5$ \\
$5$ & $(2, 1, 1, 1, 1)$ & $2^3$ \\
$6$ & $(1, 1, 1, 1, 1, 1)$ & $2^1$ \\
\bottomrule
\end{tabular}
\smallskip
\caption{Bounds in case $[K_0:\Q] = 2$.}\label{table:degrees}
\end{table}

Now, let $L_i$ denote the Galois closure of $K_i/\Q$ for each $i$, and let $\tilde{L}$ denote the compositum of $L_1, L_2, \dots, L_t$. As $K'$ is the compositum of the $K_i$, it follows that $L' = \tilde{L}$. Thus, it suffices to prove $[\tilde{L}:\Q]$ divides $2^8$; by Proposition \ref{prop:harbater}, we need only prove $[\tilde{L}:\Q] \leq 2^8$.

If $K_0 = \Q$, by Theorem \ref{theorem:Jones}, we see $[L_i:\Q]$ divides $8$ if $d_i = 4$; otherwise, $[L_i:\Q] = d_i$. Consequently,
\[ [\tilde{L}:\Q] \leq \prod_{i=1}^t [L_i:\Q] \leq 16, \]
by checking each possible list of degrees $\mathbf{d}$.

Now suppose $[K_0:\Q] = 2$. Then $[K_i:\Q] = 2d_i$, and by Theorem \ref{theorem:Jones} and Lemma \ref{lemma:octic}, we see $[L_i:\Q]$ divides $2$, $8$, or $64$, as $d_i$ is $1$, $2$, or $4$, respectively. We obtain
\[ [\tilde{L}:\Q] = [K_0:\Q][\tilde{L}:K_0] \leq 2 \cdot \prod_{i=1}^t [L_i:K_0]. \]
So, for example when $\mathbf{d} = (4,2)$, or $\mathbf{d} = (2,2,2)$, we have
\[ [\tilde{L}:\Q] \leq 2 \cdot 32 \cdot 4 = 2^8, \qquad [\tilde{L}:\Q] \leq 2 \cdot 4^3 = 2^7, \]
respectively. The remaining cases are handled by the same style of argument (all the bounds are given in Table \ref{table:degrees}).
\end{proof}
\begin{proof}[Proof of Proposition \ref{prop:Jacobian}]
Suppose $C/K_0$ is a curve of genus $2$ with Jacobian $J$, and $[J] \in \mathscr{G}^\mathrm{pp}(K_0,2,2)$. Then we have
\[ K_0 \subseteq K_0(J[2]) \subset K_0(J[2^\infty]). \]
The tower is unramified away from $2$; the Galois closure of the lower extension is a $2$-extension by the previous proposition, and the upper extension is pro-$2$ generally. Thus, $[J] \in \mathscr{A}^\mathrm{pp}(K_0,2,2)$.
\end{proof}

\subsection{Case 2: Product of elliptic curves}

The argument in this case is straightforward.
\begin{proposition}
Suppose $K_0/\Q$ is unramified away from $\{2, \infty \}$ and $[K_0:\Q] \leq 4$. Suppose $[A] \in \mathscr{G}^\mathrm{pp}(K_0,2,2)$, and $A \cong_{K_0} E_1 \times E_2$ for some elliptic curves $E_i/K_0$. Then $A$ is heavenly at $2$.
\end{proposition}
\begin{proof}
Necessarily, the curves $E_i/K_0$ have good reduction away from $2$. By Proposition \ref{prop:ell_curve}, both curves are heavenly at $2$. Thus, $K_0(E_i[2]) \subseteq \ten$, and so
\[ K_0(A[2]) = K_0(E_1[2]) \cdot K_0(E_2[2]) \subseteq \ten, \]
also.
\end{proof}

\subsection{Case 3: Restriction of scalars}

The situation in this case is slightly more delicate, because the quadratic extension introduced by the restriction of scalars could possibly ramify at a prime other than $2$. We will make use of the following lemmas, which are essentially exercises in Galois theory:
\begin{lemma}\label{lemma:gal-over-gal}
Let $E/F$ and $L/E$ be Galois extensions (inside a fixed algebraic closure of $F$), of degrees $d$ and $m$, respectively. Let $L^*/F$ denote the Galois closure of $L/F$. Then $[L^*:F] \leq d \cdot m^d$.
\end{lemma}
\begin{proof}
Set $G = \Gal(L^*/F)$, $N = \Gal(L^*/E)$, $H = \Gal(L^*/L)$. Then $H \trianglelefteq N \trianglelefteq G$. Since $[G:N] = d$, we may choose $\sigma_1, \dots, \sigma_d \in G$ representing each coset of $G/N$. Set $H_i = \sigma_i H \sigma_i^{-1}$. Then it is routine to verify that each $H_i \trianglelefteq N$. Let $L_i \subseteq L^*$ be the subfield fixed by $H_i$. These $d$ fields give every conjugate of $L/F$ within $L^*$. Consequently, $L^*$ coincides with the compositum over $E$ of the $L_i$, and each $L_i/E$ is Galois. Since $[E:F] = d$ and $[L_i:E] = m$, we obtain $[L^*:F] \leq d \cdot m^d$.
\end{proof}
We let $\mathfrak{S}_3$ denote the symmetric group on $3$ symbols.
\begin{lemma}\label{lemma:S3}
Let $F_1/F_0$ and $F_2/F_0$ be Galois $\mathfrak{S}_3$-extensions. The Galois group $G$ of the compositum $F_1 F_2/F_0$ is one of $\mathfrak{S}_3 \times \mathfrak{S}_3$, $(\Z/3\Z \times \Z/3\Z) \rtimes \Z/2\Z$, $\mathfrak{S}_3$. In every case, $G$ has a subgroup of exact index $3$.
\end{lemma}
\begin{proof}
Since $\mathfrak{S}_3$ has no normal subgroups of index $3$, the (necessarily Galois) extension $F_1 \cap F_2/F_0$ has degree $1$, $2$, or $6$. The result follows by considering each possibility in turn.
\end{proof}

\begin{proposition}
Suppose $K_0/\Q$ is unramified away from $\{2, \infty \}$ and $[K_0:\Q] \leq 2$. Suppose $[A] \in \mathscr{G}^\mathrm{pp}(K_0, 2, 2)$. If $A$ is isomorphic over $K_0$ to $\Res_{K'/K_0} E_1$ for some elliptic curve $E_1/K'$, then $A$ is heavenly at $2$.
\end{proposition}
\begin{proof}
Let $E_2/K'$ be the Galois twist of $E_1$ with respect to the unique nontrivial element of $\Gal(K'/K_0)$.
As $A/K_0$ has good reduction away from $2$, the same is true for $E_i/K'$, $i=1,2$. If $K'/K_0$ is unramified away from $\{2, \infty \}$, then $K'/\Q$ is a quartic extension also unramified away from $\{2, \infty \}$, and so by Proposition \ref{prop:ell_curve}, each $E_i$ is heavenly at $2$. Thus, the Galois closure of $K'(E_i[2])/\Q$ is a $2$-extension. Since $K_0(A[2])$ is contained in the compositum of these two extensions, the Galois closure of $K_0(A[2])/\Q$ is also a $2$-extension, and so $[A] \in \mathscr{A}^\mathrm{pp}(K_0,2,2)$.

Now, suppose that some prime $\mathfrak{p}$ of $K'$ ramifies in $K'/K_0$, with $\mathfrak{p} \nmid 2\mathcal{O}_{K'}$. Notice that we have the equality of fields
\[ M := K' \cdot K_0(A[2]) = K'(A[2]) = K'(E_1[2] \cup E_2[2]) = K'(E_1[2]) \cdot K'(E_2[2]). \]
The two extensions $K'(E_i[2])/K'$ must have isomorphic Galois groups. Choose $\Gamma \leq \mathfrak{S}_3$ such that $\Gamma \cong \Gal(K'(E_i[2])/K')$, and let $c = |\Gamma|$. We must have that $[M:K']$ divides $c^2$. Moreover, since $K'/K_0$ and $K_0(A[2])/K_0$ are both Galois extensions, $M$ is in fact Galois over $K_0$.
\[
\begin{tikzcd}[row sep=2ex]
{} & M \arrow[-, swap]{dl}{2} \arrow[-]{dd} \\
K_0(A[2]) & \\
& K' \\
K_0 \arrow[-]{uu} \arrow[-, swap]{ur}{2} &
\end{tikzcd}
\]
Since $M/K_0$ is Galois, we see $[M:K_0(A[2])]$ divides $2 = [K':K_0]$, and $[M:K']$ divides $[K_0(A[2]):K_0]$. But since $K'/K_0$ is ramified at $\mathfrak{p}$ and $K_0(A[2])/K_0$ is not ramified at $\mathfrak{p}$, the extension $M/K_0(A[2])$ must reflect this ramification; it cannot be a trivial extension. Thus, $[M:K_0(A[2])] = 2$ and $\Gal(M/K') \cong \Gal(K_0(A[2])/K_0)$.

We claim $c \neq 6$. For otherwise, Lemma \ref{lemma:S3} guarantees that $\Gal(K_0(A[2])/K_0)$ possesses a subgroup of index $3$. Such a subgroup corresponds to a cubic extension $L/K_0$ with $L \subset K_0(A[2])$. Consequently, $L/\Q$ is a degree $6$ extension, unramified away from $\{2, \infty \}$, which contradicts Theorem \ref{theorem:Jones}.

So $c < 6$, which in fact implies $c \leq 3$ and $[K_0(A[2]):K_0] \leq 9$. Thus, by Lemma \ref{lemma:gal-over-gal}, the Galois closure of $K_0(A[2])/\Q$ has degree at most $2 \cdot 9^2 = 162$. By Proposition \ref{prop:harbater}, the Galois closure must be a $2$-extension, and so $[A] \in \mathscr{A}^\mathrm{pp}(K_0,2,2)$.
\end{proof}

\section{Failure in higher dimensions}

At this point, one might na\"{\i}vely guess that in the special case $k = \Q$, $\ell = 2$, $\mathscr{A}^\mathrm{pp}(\Q,g,2) = \mathscr{G}^\mathrm{pp}(\Q,g,2)$ for all $g \geq 1$. This is not the case, as we now show via Weil restriction.

First, we briefly review the useful Example 2.24 of \cite{Harbater:1994}. Let $\zeta \in \bar{\Q}$ be a primitive $64$-th root of unity, and set $\eta_0 = \zeta^{16}(\zeta + \zeta^{-1})$ and $F_0 = \Q(\eta_0)$. Let $L_0$ be the Hilbert class field of $F_0$. Harbater demonstrates that $L_0/\Q$ is Galois with Galois group $\Z/17\Z \rtimes (\Z/17\Z)^\times$, and $L_0/\Q$ is a (unique) Galois extension of degree $272$ unramified away from $2$.

We now demonstrate a non-heavenly element of $\mathscr{G}^\mathrm{pp}(\Q,272,2)$.
\begin{proposition}
The set $\mathscr{A}^\mathrm{pp}(\Q,272,2)$ is a proper subset of $\mathscr{G}^\mathrm{pp}(\Q,272,2)$. That is, there exists an abelian variety $A/\Q$ of dimension $272$ with good reduction away from $2$ which is not heavenly at $2$.
\end{proposition}
\begin{proof}
Select an elliptic curve $E/\Q$ such that $[E] \in \A(\Q,1,2)$ and $\Q(E[2]) = \Q$. There are two such curves, up to $\Q$-isomorphism, given by Cremona's labeling as `32a2' and `64a1':
\[ \mathrm{(32a2)}\quad y^2 = x^3 - x, \qquad \qquad \mathrm{(64a1)} \quad y^2 = x^3 - 4x. \]
Let $L_0$ be as in the previous paragraph, and set
\[ A := \Res_{L_0/\Q} (E \times_\Q L_0). \]
This is an abelian variety defined over $\Q$, and $[A] \in \mathscr{G}^\mathrm{pp}(\Q,272,2)$.
(Note that $A$ is principally polarized by \cite[Proposition 2]{Diem-Naumann:2003}.)
When we view $A[2]$ as a $G_\Q$-module, we have
\[ A[2] \cong \Ind_{G_{L_0}}^{G_\Q} E[2]. \]
Let $H = \Gal(L_0/\Q)$. Viewed only as an abelian group, we have
\[ \Ind_{G_{L_0}}^{G_\Q} E[2] = \bigoplus_{\sigma \in H} E[2] \cong \bigoplus_{\sigma \in H} \F_2^{\oplus 2}. \]
However, $E[2] \subseteq E(\Q)$, and as $G_\Q$-modules, we have $E[2] \cong_{G_\Q} \F_2^{\oplus 2}$, where $G_\Q$ acts on $\F_2^{\oplus 2}$ trivially. Thus the action of $G_\Q$ on $\Ind_{G_{L_0}}^{G_\Q} E[2]$ is given simply by the permutation action on the summands indexed by $\sigma \in H$. From this we conclude
\[ L_0 = L_0 \cdot \Q(E[2]) \subseteq \Q(E[2])\bigl( A[2] \bigr) = \Q(A[2]). \]
Since $[L_0:\Q] = 272$, we have $\Q(A[2^\infty]) \not\subseteq \ten$, and so $[A] \not\in \mathscr{A}^\mathrm{pp}(\Q,272,2)$.
\end{proof}

\bibliographystyle{halpha}
\bibliography{Ras-Tam-bib}

\end{CJK}
\end{document}